\documentclass[a4paper,reqno]{amsart}

\usepackage{graphicx}
\usepackage{mathrsfs}
\usepackage{amsfonts}
\usepackage{amssymb}
\usepackage{amsmath}
\usepackage{amsthm}
\usepackage{amscd}
\usepackage[all,2cell]{xy}
\usepackage{color}
\usepackage[pagebackref,colorlinks]{hyperref}
\usepackage{enumerate}
\usepackage{smartdiagram}

 \usepackage{wrapfig}
\usepackage{tikz}
\usetikzlibrary{shapes,shadows,arrows}
\usepackage{tkz-graph}

\usepackage{comment}
\usepackage{hyperref}
\UseAllTwocells \SilentMatrices

\numberwithin{equation}{section}

\usepackage{comment}

\theoremstyle{plain}
\newtheorem{thm}{Theorem}[section]
\newtheorem{prop}[thm]{Proposition}

\newtheorem{lemma}[thm]{Lemma}
\newtheorem{conjecture}[thm]{Conjecture}

\theoremstyle{definition}
\newtheorem{deff}[thm]{Definition}
\newtheorem{example}[thm]{Example}
\theoremstyle{remark}

\newcommand{\dom}{\mathop{\boldsymbol d}}
\newcommand{\ran}{\mathop{\boldsymbol r}}
\newcommand{\inv}{^{-1}}

\def\N{\mathbb N}

\newcommand{\Ima}{\operatorname{Im}}

\def\g{\gamma}
\def\d{\delta}
\def\G{\Gamma}
\def\mG{\mathcal{G}}
\def\mGo{\mG^{(0)}}
\def\mH{\mathcal{H}}
\def\mHo{\mH^{(0)}}

\newcommand{\FG}{[[\mG]]}

\def\xra{\xrightarrow[]{}}
\def\a{\alpha}
\def\b{\beta}

\def\O{\mathcal{O}}

\def\SS{\mathcal{S}}

\def\sub{\subseteq}

\def\M{\mathbb{M}}

\def \Z{\mathbb Z}
\def\-{\text{-}}

\def\LL{\mathcal{L}}

\newcommand{\supp}{\operatorname{supp}}
\newcommand{\Iso}{\operatorname{Iso}}

\newcommand{\gr}{\operatorname{gr}}

\newcommand{\id}{\operatorname{id}}

\begin{document}

\title[\'Etale groupoids and Steinberg algebras]{\'Etale groupoids and Steinberg algebras\\ a concise introduction}

\author {Lisa Orloff Clark}
\address {Lisa Orloff Clark: School of Mathematics and Statistics, Victoria University of Wellington, NEW ZEALAND}
\email {lisa.clark@vuw.ac.nz}

\author{Roozbeh Hazrat}
\address{Roozbeh Hazrat: 
Centre for Research in Mathematics\\
Western Sydney University\\
AUSTRALIA} \email{r.hazrat@westernsydney.edu.au}

\subjclass[2010]{22A22, 18B40,16D25}

\keywords{Groupoid, Inverse semigroup, Steinberg algebra, Leavitt path algebra}

\date{\today}

\begin{abstract} 

We give a concise introduction to (discrete) algebras arising from \'etale groupoids, (aka Steinberg algebras) and describe their close relationship with groupoid $C^*$-algebras. Their connection to partial group rings via inverse semigroups also explored. 

\end{abstract}

\maketitle

 
\begin{flushright}
{\it Keep fibbing and you'll end up with the truth!} \\ Dostoyevsky, Crime and Punishment. 
\end{flushright}

\section{Introduction}

\tikzstyle{decision} = [diamond, draw, fill=blue!50]
\tikzstyle{line} = [draw,  -stealth, thick]
\tikzstyle{lined} = [draw, bend right=45, thick, dashed]
\tikzstyle{elli}=[draw, ellipse, top color=white, bottom color=red!90 ,minimum height=8mm, text width=6.3em, text centered]

\tikzstyle{elli2}=[draw, ellipse, top color=white, bottom color=green!90 ,minimum height=8mm, text width=6.3em, text centered]

\tikzstyle{elli3}=[draw, ellipse, top color=white, bottom color=blue!90 ,minimum height=8mm, text width=6.3em, text centered]

\tikzstyle{block} = [draw, rounded corners, rectangle, top color=white, bottom color=white!80, text width=8em, text centered, minimum height=15mm, node distance=7em]

\tikzstyle{block2} = [draw, rounded corners, rectangle, top color=white!80, bottom color=white, text width=8em, text centered,  minimum height=15mm, node distance=7em]

\begin{wrapfigure}{r}{0.45\textwidth} 

\begin{tikzpicture}[scale=0.7, transform shape]

\GraphInit[vstyle = Shade]

\node [block] (groupoid) {\bf  \'etale groupoid};

\node [elli3, above of=groupoid, xshift=0em, yshift=9em] (gdynamic) {\bf  groupoid dynamics};

\node[block2, below of=groupoid] (partial) {\bf \footnotesize partial inverse semigroup action};

\node [elli3, below of=partial, xshift=0em, yshift=-9em] (gpartial) {\bf  semigroup dynamics};

\tikzset{
  EdgeStyle/.append style = {->} 
  }

\Edge (groupoid)(gdynamic)

\Edge (partial)(gpartial)

\node[elli, above of=groupoid, xshift=-12em, yshift=3em] (steinberg) {\footnotesize  \bf Steinberg algebra};

\node[elli2, below of=partial, xshift=-12em, yshift=-3em] (ring) {\footnotesize  \bf partial group ring};

\node[elli, above of=groupoid, xshift=12em, yshift=3em] (cstar) {\footnotesize \bf groupoid $C^*$-algebra};

\node[elli2, below of=partial, xshift=12em, yshift=-3em] (crossp) {\footnotesize \bf partial crossed product};

\tikzset{
  EdgeStyle/.append style = {->, bend left} 
  }

\Edge (groupoid)(steinberg)
\tikzset{
  EdgeStyle/.append style = {->, bend right} 
  }

\Edge (groupoid)(cstar)
\Edge (partial)(ring)
\tikzset{
  EdgeStyle/.append style = {->, bend left} 
  }
\Edge (partial)(crossp)
\tikzset{
  EdgeStyle/.append style = {->, bend left} 
  }
\Edge (partial)(crossp)
\Edge (groupoid)(partial)
\Edge (partial)(groupoid)

\end{tikzpicture}

\end{wrapfigure}

In the last couple of years, \'etale groupoids have become a focal point in of several areas of mathematics. The convolution algebras arising from \'etale groupoids, considered both in analytical setting \cite{renault} and algebraic setting \cite{cfst,st}, include many deep and important examples such as Cuntz algebras~\cite{cuntz1} and Leavitt algebras~\cite{vitt57} and allow systematic treatment of them.  Partial actions and partial symmetries can also be realised as \'etale groupoids (via inverse semigroups), allowing us to relate convolution algebras to partial crossed products (\cite{de,exel20081}). 

Realising that the invariants long studied in topological dynamics can be modelled on \'etale groupoids (such as homology, full groups and orbit equivalence~\cite{matui2012}) and that these are directly related to invariants long studied in analysis and algebra (such as $K$-theory) allows an interaction between areas;  we can use techniques developed in algebra in analysis and vice versa.  The \'etale groupoid is the Rosetta stone. 

The study of representations of \'etale groupoids on Hilbert space and the associated $C^*$-algebras was pioneered by Renault in~\cite{renault}. In this seminal work he showed that Cuntz algebras can be realised using groupoid machinery. In \cite{kprr} the authors associated an \'etale groupoid to a directed graph and the subject of graph $C^*$-algebras was born.  The universal construction of these graph  $C^*$-algebras via generators and relations was then established in \cite{bates}. The analytic activities then exploded in several directions; to describe the properties of the graph $C^*$-algebras directly from the geometry of the graph, to classify these algebras and to extend the definition to other type of graphs (such a higher rank graphs~\cite{kphigher}).

There has long been a trend of ``algebraisation'' of concepts from operator theory into a purely algebraic context. This seems to have started with von Neumann and Kaplansky who devised ways of seeing operator algebraic properties in underlying discrete structures \cite{kap}. As Berberian puts it in \cite{berber}, ``if all the functional analysis is stripped away...what remains should stand firmly as a substantial piece of algebra, completely accessible through algebraic avenues''.

This translation did happen in the setting of graph algebras in the reverse order (and with about 30 year lag): in \cite{abrams2005, ara2006} the algebraic analogue of graph $C^*$-algebras were defined using generators and relations (called Leavitt path algebras) and then the algebraic analogue of groupoid $C^*$-algebras was developed in \cite{cfst,st} (now called Steinberg algebras).

Another strand is the work of Exel on partial actions of groups on spaces and their corresponding $C^*$-algebras. Again, the algebraic version of this theory is developed~\cite{de} and the close connections with groupoids is established~\cite{exel20081,batista}. 

This survey exclusively concentrates on \'etale groupoids with totally disconnected unit spaces (aka ample groupoids)  and their convolution algebras (aka, Steinberg algebras). One reason is that over such groupoids our $R$-algebras are just $R$-valued continuous 
 functions with compact support over the groupoid and there is a known universal description for such algebras. We will briefly describe the situation when the groupoid is not Hausdorff as well. 
 We describe their connections with groupoid $C^*$-algebras and Exel's partial constructions.   The concepts of inverse semigroup and groupoid are tightly related (as the diagram in the first page shows) and are models for partial symmetries. In Sections \ref{inverse4} and \ref{groupoid4} we study these concepts with a view towards the algebras that arise from them which we describe in Section~\ref{stein4}.

The use of groupoids extends to  many areas of mathematics, from ergodic theory and functional analysis (such as work of Connes in noncommutative geometry~\cite{connes}) to homotopy theory (\cite{brown1}), algebraic geometry, differential geometry and group theory. The reader is encouraged to consult \cite{brownbang,higgins,weinstein} for more details on the history and development of groupoids.

\section{Inverse semigroups}\label{inverse4}

There is a tight relation between the notion of groupoids and its ``dual'' inverse semigroups. We start the survey with a description of inverse semigroups. 

\subsection{Inverse semigroups} \label{unfairjfgng}

Recall that a semigroup is a nonempty set with an associative binary operation. For a semigroup $S$,  the element $x\in S$ is called \emph{regular}  if $xyx=x$. In this case we can arrange that $xyx=x$ and $yxy=y$ and we say $x$ has an \emph{inner inverse}. We say a semigroup is \emph{regular} if each element has an inner inverse. 

An inverse semigroup is a semigroup that each element has a unique inner inverse. Namely, an \emph{inverse semigroup} is a semigroup $S$ such that, for each $s\in S$, there exists a unique element $s^*\in S$ such that  
\[ss^*s=s   \text{ and } s^*ss^*=s^*.\]  The uniqueness guarantees that the map $s \rightarrow  s^*$ induces an involution on $S$. One can check that 
\[E(S):=\{ ss^* \mid  s\in S\},\] is the set of idempotents of $S$ and is an abelian subsemigroup. One way to prove a semigroup is an inverse semigroup is to show that it is a regular semigroup  and the set of idempotents are abelian. 

In fact $E(S)$ is a meet semilattice with respect to the partial ordering $e\leq f$ if $ef=e$; the meet is the product.  The partial order extends to the entire inverse semigroup by putting $s\leq t$ if $s=te$ for some idempotent $e\in E(S)$ (or, equivalently, $s=ft$ for some  $f\in E(S)$).  This partial order is preserved under multiplication and inversion.

Most of the inverse semigroups we encounter have a zero element. An inverse semigroup $S$ has a zero element $0$ if  $0x=0=x0$ for all $x\in S$.  The zero element is unique when it exists and often corresponds to the empty set in our concrete examples. Any semigroup homomorphism $p\colon S\rightarrow T$ of inverse semigroups automatically preserves the involution, i.e., $p(s^*)=p(s)^*$.

Parallel to the group of symmetries and the theorem of Cayley, we next define the inverse semigroup of partial symmetries and recall the theorem of Wagner-Preston. Let $X$ be a set and $A,B \subseteq X$.  A bijective map $f: A\rightarrow B$ is called a \emph{partial symmetry} of $X$.  Denote by $\mathcal I(X)$ the collection of all partial symmetries of $X$. The set $\mathcal I(X)$ is an inverse semigroup with zero under the operation given by composition of functions in the largest domain in which the composition may be defined. The zero element corresponds to the map assigned to an empty set. The  Wagner-Preston theorem guarantees that any inverse semigroup is a subsemigroup of $\mathcal I(X)$ for some set $X$.

A majority of inverse semigroups we encounter here are naturally `graded'. If $S$ is an inverse semigroup with possibly $0$ and $\Gamma$ is a discrete group, then $S$ is called a $\Gamma$-\emph{graded inverse semigroup} if there is a map $c : S \setminus\{0\} \to \Gamma$ such that $c(st) = c(s)c(t)$,
whenever $st \not= 0$.  For $\gamma \in \Gamma$, if we set $S_\gamma:=c^{-1}(\gamma)$, then $S$
decomposes as a disjoint union
\[ S \setminus\{0\} = \bigsqcup_{\g \in \G} S_\g, \]
and we have  $S_\b S_\g \subseteq S_{\b\g}$, if the product is not zero. 
We say that $S$ is
\emph{strongly graded} if $S_\b S_\g =S_{\b\g}$, for all $\b,\g$. 
The reader is referred to Mark Lawson's book~\cite{Lawson} for the theory of inverse semigroups.

\subsection{Examples of inverse semigroups} 

Clearly any group is an inverse semigroup without zero unless it is a trivial group. The Theorem of  Wagner-Preston shows that the partial symmetries are the `mothers' of all inverse semigroups. 

\begin{example}[\bf {Graph inverse semigroups}]\label{ingfgfgfhr1}
Directed graphs provide concrete examples for constructing a variety of combinatorial structures, such as semigroups, groupoids and algebras. We briefly recall the definition of a directed graph and the construct the first combinatorial structure out of them, namely, graph inverse semigroups. 

A \emph{directed graph} $E$ is a tuple $(E^{0}, E^{1}, r, s)$, where $E^{0}$ and $E^{1}$ are
sets and $r,s$ are maps from $E^1$ to $E^0$. We think of each $e \in E^1$ as an arrow
pointing from $s(e)$ to $r(e)$. We use the convention that a (finite) path $p$ in $E$ is
a sequence $p=\a_{1}\a_{2}\cdots \a_{n}$ of edges $\a_{i}$ in $E$ such that
$r(\a_{i})=s(\a_{i+1})$ for $1\leq i\leq n-1$. We define $s(p) = s(\a_{1})$, and $r(p) =
r(\a_{n})$. If $s(p) = r(p)$, then $p$ is said to be closed. If $p$ is closed and
$s(\a_i) \neq s(\a_j)$ for $i\neq j$, then $p$ is called a \emph{cycle}. An edge $\a$ is an exit
of a path $p=\a_1\cdots \a_{n}$ if there exists $i$ such that $s(\a)=s(\a_i)$ and
$\a\neq\a_i$. A graph $E$ is called \emph{acyclic} if there are no closed path in $E$. For a path $p$, we denote by $|p|$ the length of $p$, with the convention that  $|v|=0$.

A directed graph $E$ is said to be \emph{row-finite} if for each vertex $u\in E^{0}$,
there are at most finitely many edges in $s^{-1}(u)$. A vertex $u$ for which $s^{-1}(u)$
is empty is called a \emph{sink}, whereas $u\in E^{0}$ is called an \emph{infinite
emitter} if $s^{-1}(u)$ is infinite. If $u\in E^{0}$ is neither a sink nor an infinite
emitter, then it is called a \emph{regular vertex}.

\begin{deff} \label{definverlp} Let $E=(E^{0}, E^{1}, r, s)$ be a directed graph.  The \emph{graph inverse semigroup} $S_E$ 
is the semigroup with zero generated by the sets $E^0$ and $E^1$, together with a set $E^*= \{e^* \mid e \in E^1\}$,
satisfying the following relations:
\begin{itemize}
\item[(0)] $uv=\delta_{u, v}v$ for every $u, v\in E^{0}$;
\item[(1)] $s(e)e=er(e)=e$ for all $e\in E^{1}$;
\item[(2)] $r(e)e^{*}=e^{*}=e^{*}s(e)$ for all $e\in E^{1}$;
\item[(3)] $e^{*}f=\delta_{e, f}r(e)$ for all $e, f\in E^{1}$.
\end{itemize}
\end{deff}

For a path $p=e_1e_2\dots e_n$, denoting $p^*=e_n^*\dots e_2^* e_1^*$, one can show that elements of $S_E$ are of the form $pq^*$ for some paths $p$ and $q$ and the unique inner inverse of $pq^*$ is $qp^*$. 

This definition was first given in \cite{ash} and then in \cite{Patersonsemi} in relation with groupoids and groupoids $C^*$-algebras. The fact that Definition~\ref{definverlp} gives an inverse semigroup was checked in details in \cite[Propositions~3.1, 3.2]{Patersonsemi}. The graph inverse semigroup associated to a graph with one vertex and $n$ loops, is called \emph{Cuntz inverse semigroup} and it was defined in \cite[p. 141]{renault}.  We remark that  that the universal groupoid of $S_E$ (see \cite{st}) is the graph groupoid $\mG_E$ which will be studied in~\S\ref{gfgfgfg1}.

For a graph $E$, the inverse semigroup $S_E$ has a natural $\mathbb Z$-graded where $c(pq^*)=|p|-|q|$. We also refer the reader to \cite{zak} for further study on these inverse semigroups. 
\end{example}

\begin{example}[\bf{Exel's inverse semigroup associated to a group}]\label{exhdgtrhdf}
Any group is an inverse semigroup. In  \cite{exel1998}, Exel defined a semigroup $S(G)$ associated to the partial actions of the group $G$ on sets (Example~\ref{exphyhyh5}) and proved that this semigroup is in fact an inverse semigroup. He then established that the partial actions of $G$ on a set $X$ are in one-to-one correspondence with the action of $S(G)$ on $X$.  As the construction of $S(G)$ is very natural we give it here. 

Let $G$ be a group with unit $\varepsilon$. We define  $S(G)$ to be the semigroup generated by $\{[g]\;|\; g\in G\}$ subject to the following relations: for $g, h\in G$:
\begin{itemize}

\item [(i)] $[g^{-1}][g][h] = [g^{-1}][gh]$; 

\item[(ii)] $[g][h][h^{-1}] = [gh][h^{-1}]$; and

\item[(iii)] $[g][\varepsilon] = [g]$.
\end{itemize}

Observe that $[\varepsilon][g]=[gg^{-1}][g]=[g][g^{-1}][g]=[g][g^{-1}g]=[g][\varepsilon]=[g]$. Then $S(G)$ is a semigroup with unit $[\varepsilon]$. 
It was proved in \cite[Theorem~3.4]{exel1998} that $S(G)$ is an inverse semigroup and each element of $x\in S(G)$ can be written uniquely as 
$x=[t_1][t_1^{-1}]\dots [t_r][t_r^{-1}][g]$. This gives that $S(G)$ is also a $G$-graded inverse semigroup. 

Further in~\cite{bussexel} Buss and Exel showed that starting from an inverse semigroup $G$, a similar construction as above (replacing $g^{-1}$ by $g^*$) is also an inverse semigroup. 

\end{example}

\section{Groupoids}\label{groupoid4}

\subsection{Groupoids} 

The use of groupoids to study structures whose operations are partially defined is firmly recognised~\cite{brown1,higgins,Lawson, weinstein}. We start by recalling the definition of a groupoid with a suitable topology, {\it i.e.}, an ample groupoid.  We will eventually describe a ring of $R$-valued continuous functions on an ample groupoid, where $R$ is a (commutative, unital) ring.   These are the main objects of this survey, namely Steinberg algebras.

A \emph{groupoid} is a small category in which every morphism is invertible. It can also be viewed as a generalization of a group which has a partially defined binary operation.  
Let $\mG$ be a groupoid. If $x\in\mG$, $\dom(x):=x^{-1}x$ is the \emph{domain} of $x$ and $\ran(x):=xx^{-1}$ is
its \emph{range}. Thus the pair $(x,y)$ in the category $\mG$ is composable if and only if $\ran(y)=\dom(x)$ and in this case $xy \in \mG$. Denote 
$\mG^{(2)} := \{(x,y) \in \mG \times \mG : \dom(x) = \ran(y)\}$. The set
$\mG^{(0)}:=\dom(\mG)=\ran(\mG)$ is called the \emph{unit space} of $\mG$. Note that we identify the objects of the category $\mG$ with $\mGo$. 
which are the identity morphisms of the category $\mG$ in the sense that $x\dom(x)=x$ and $\ran(x)x=x$ for all $x \in \mG$. 

The collection of morphisms whose domain and range are a fixed unit $u \in \mGo$ is a group and the collection of all of these groups is called the isotropy bundle $\operatorname{Iso}(\mG)$, that is, 
\[\operatorname{Iso}(\mG) : = \{\gamma \in \mG : d(\gamma) = r(\gamma)\}.\]

For subsets 
$U,V\subseteq \mG$, we define
\begin{equation}\label{pofgtryhf}
    UV=\big \{ x y  \mid x \in U, y \in V \text{ and } \dom(x)=\ran(y) \big\},
\end{equation}
and
\begin{equation}\label{pofgtryhf2}
U^{-1}=\big \{x^{-1} \mid x \in U \big \}.
\end{equation}

If $\mG$ is a groupoid and $\Gamma$ is a group, then $\mG$ is called a $\Gamma$-\emph{graded groupoid} if there is functor $c:\mG \rightarrow \Gamma$, {\it i.e.,} there is a function $c : \mG \to \G$ such that $c(x)c(y) = c(x y )$ for all $(x,y)
\in \mG^{(2)}$. For $\gamma \in \Gamma$, if we set $\mG_\gamma:=c^{-1}(\gamma)$, then $\mG$
decomposes as a disjoint union
\[ \mG= \bigsqcup_{\g \in \G} \mG_\g, \]
and we have  $\mG_\b \mG_\g \subseteq \mG_{\b\g}$. 
We say that $\mG$ is
\emph{strongly graded} if $\mG_\b \mG_\g =\mG_{\b\g}$, for all $\b,\g$. For $\g
\in \G$, we say that $X\subseteq \mG$ is $\g$-graded if $X\subseteq \mG_\g$. We
have $\mG^{(0)} \subseteq \mG_\varepsilon$, so $\mG^{(0)}$ is
$\varepsilon$-graded, where $\varepsilon$ is the identity of the group $\Gamma$. Graded groupoids were studied in ~\cite{chr}.

\subsection{Topological groupoids}  A topological groupoid is a groupoid endowed with a topology under which the inverse map
is continuous, and composition is continuous with respect to the relative product
topology on $\mG^{(2)}$.   An \emph{\'etale} groupoid is a 
topological groupoid $\mG$ such that the
domain map $d$ is a local homeomorphism. In this case, the range map $r$ is also a local
homeomorphism.
An \emph{open bisection} of $\mG$ is an open subset $U\subseteq \mG$ such
that $\dom|_{U}$ and $\ran|_{U}$ are homeomorphisms onto an open subset of $\mG^{(0)}$. Notice that a groupoid is \'etale if and only if it has a basis of open bisections.  

We say that a topological groupoid $\mG$ is \emph{ample} if there is a basis of compact open bisections.  An ample groupoid is automatically \'etale, locally compact and $\mG^{(0)}$ is an open subset of $\mG$.  The terminology in the literature is inconsistent:  sometimes the term `\'etale' also includes the assumptions of local compactness and $\mGo$ Hausdorff.   
We will focus on the situation where  $\mG$ is Hausdorff ample so these two assumptions are automatically true.  We discuss non-Hausdorff groupoids briefly in Section~\ref{sec:nonH}.

In the topological setting, we call a groupoid $\mG$, a $\G$-graded groupoid, if
the functor $c : \mG \to \G$ is continuous with respect to the discrete topology on $\Gamma$; such a function $c$ is called a \emph{cocycle} on $\mG$.

\begin{lemma} \label{fghfhfhuyt} 
Let $\mG$ be an \'etale groupoid.
If $\mGo$ is a finite, then $\mG$ is a discrete topological space. 
\end{lemma}
\begin{proof}
Since $\mGo$ is finite and Hausdorff, it is discrete.    
Fix $\gamma \in \mG$.  We show that $\{\gamma\}$ is open. 
Since $\mG$ is \'etale, $\gamma$ is contained in an open bisection $U$.
Also, since $r$ is continuous, $r^{-1}(\{r(\gamma)\})$ is open.
But $r$ is injective on $U$ so 
\[r^{-1}(\{r(\gamma)\}) \cap U = \{\gamma\}.\qedhere\]
\end{proof}

We will determine the Steinberg algebra associated to finite $\mGo$ in Proposition~\ref{gfhgdhr22}.

\subsection{Examples of groupoids} 
\begin{example}[\bf {Transitive groupoids}]\label{exphyhyh}  A groupoid is called \emph{connected} or \emph{transitive} if for any $u,v \in \mGo$, there is a $x\in \mG$ such that $u=s(x)$ and $r(x)=y$. 

Let $G$ be a group and $I$ a non-empty set. The set $I\times G \times I$, considered as morphisms, forms a groupoid where the composition defined by $(i,g,j)(j,h,k)=(i,gh,k)$. One can show that this is a transitive groupoid and any transitive groupoid is of this form (\cite[Ch.~3.3, Prop.~6]{Lawson}). If $I=\{1,\dots, n\}$, we denote $I\times G \times I$ by $n\times G \times n$. Note that this groupoid is naturally strongly $G$-graded. This seems to be the first appearance of groupoids after they were introduced by Brandt in 1926 \cite{brandt} (see  in \cite{brownbang} for a nice history of groupoids). 
\end{example}

In the next three examples we explore how a group action on a (combinatorial) structure can be naturally captured by a groupoid. The first example is the action of a group on a set, and we then continue with a partial action of a group and inverse semigroup on a set. Although the (partial) action of an inverse semigroup on a set would be the most general case covering the previous two examples, for pedagogical reason we introduce these step by step. 

\begin{example}[\bf {Transformation groupoid arising from a group action}]\label{exphyhyh2}
Let $G$ be a group acting on a set $X$, {\it i.e.,}  there is a group homomorphism $G\xra {\rm \Iso}(X)$, where ${\Iso}(X)$ consists of bijective maps from $X$ to $X$ which is a group with respect to composition. Let 
\begin{equation}\label{globiala}
\mG=G \times X
\end{equation}
 and define the groupoid structure: 
$(g, h y)\cdot(h,y)=(g h ,y)$, 
and $(g,x)^{-1}=(g^{-1}, g x)$. 
Then $\mG$ is a groupoid,  
called the \emph{transformation groupoid} 
arising from the action of $G$ on $X$ (for short, $G \curvearrowright X$). 
The unit space $\mG^{(0)}$ is canonically identified with $X$ 
via the map $(\varepsilon,x)\mapsto x$. The natural cocycle $\mG\rightarrow G, (g,x)\mapsto g$, makes $\mG$ a strongly $G$-graded groupoid. Note that the range and source map would distinguish an element of this groupoid up to the stabiliser. Namely, $\ran(g,x)=x=\ran(h,x)$ and $\dom(g,x)=gx=hx=\dom(h,x)$. But when we consider the grading then we can distinguish these elements as well. 
When $X$ is a Hausdorff topological space and $G$ is a discrete group, then $\mG$ is an \'etale topological groupoid with respect to the product topology.  If, in addition, $X$ has a basis of compact open sets, then $\mG$ is ample.
\end{example}

\begin{example}[\bf {Transformation groupoid arising from a partial group action}]\label{exphyhyh5}

A \emph{partial action} of a group $G$ on a set $X$ is a data $\phi=(\phi_{g}, X_g, X)_{g\in G}$, where for each $g\in G$, $X_g$ is a subset of $X$ and $\phi_g:X_{g^{-1}}\xra X_g$ is a bijection such that 
\begin{itemize}
\item[(i)]$X_{\varepsilon} =X$ and $\phi_{\varepsilon}$ is the identity on $X$, where $\varepsilon$ is the identity of the group $G$;

\item[(ii)] $\phi_g(X_{g^{-1}}\cap X_h) = X_g\cap X_{gh}$ for all $g, h\in G$;

\item[(iii)] $\phi_g(\phi_h(x)) =\phi_{gh}(x)$ for all $g, h\in G$ and $x\in X_{h^{-1}}\cap X_{h^{-1}g^{-1}}$.
\end{itemize}

Although the above construction gives a well-define map $\pi: G \rightarrow \mathcal I(X), g \mapsto \phi_g$, this map is not a homomorphism, {\it i.e.,} $\phi_{gh}\not= \phi_g \phi_h$. However there is a one-to-one correspondence between these partial actions and the actions of inverse semigroup $S(G)$ on $X$ (see Example~\ref{exhdgtrhdf})

 Let $\phi=(\phi_g, X_g, X)_{g\in G}$ be a partial action of $G$ on $X$.   Consider the $G$-graded groupoid \begin{equation}
 \label{groupoid}
 \mG_\phi=\bigcup_{g \in G} g \times X_g,
 \end{equation} whose composition and inverse maps are given by $(g,x)(h,y)=(gh,x)$ if $y=\phi_{g^{-1}}(x)$ and $(g,x)^{-1}=(g^{-1},\phi_{g^{-1}}(x))$. Here the range and source maps are given by $\ran(g,x)=(\varepsilon, x)$, $\dom(g,x)=(\varepsilon, \phi_{g^{-1}}(x))$ with $\varepsilon$ the identity of $G$.  The unit space of $\mG_\phi$ is identified with $X$. 
 
 In case that $X$ is a topological space, we assume $X_g\sub X$ is an open subset and each $\phi_g:X_{g^{-1}}\xra X_g$ is a homeomorphism, for $g\in G$. In order to obtain an ample groupoid, we further assume that $X$ is a Hausdorff topological space that has a basis of compact open sets,  each $X_g$ is a clopen subset of $X$, and $G$ is a discrete group. 
 The topology of $\mG_\phi$  which inherited from the product topology $G\times X$ gives us an Hausdorff ample groupoid. 


In fact one can further generalise this to the setting of partial action of an inverse semigroup on sets, topological spaces and rings. In \S\ref{cross4} we will relate partial inverse semigroup rings coming out of this partial actions to Steinberg algebras.

\end{example}

\begin{example}[\bf {Transformation groupoid arising from an inverse semigroup action; groupoid of germs}]\label{exphyhyh7}

We start with a more concrete example of the groupoid of germs and then move to a more abstract construction of the groupoid of germs of an inverse semigroup acting on a space. In the topological setting, these are one of the main sources of \'etale groupoids.

Let $X$ be a nonempty set and let $S=\mathcal I(X)$ be the inverse semigroup of partial symmetries. The $S$-germ is a pair $(s,x)\in S\times X$, where $x\in \dom(s)$,  modulo the equivalence relation of germs $(s,x) \sim (t,y)$ if $x=y$ and the restriction of $s$ and $t$ coincides on a subset containing $x$. The groupoid operations defined by 
\[(s, t y)(t,y)=(st,y), \, \, (s,x)^{-1}=(s^{-1},sx).\] 

In a more abstract setting, let $S$ be an inverse semigroup acting on a set $X$, {\it i.e.,} there is a semigroup homomorphism $S\rightarrow \mathcal I(X)$.  Let 
\begin{equation}\label{nearneahdu}
\mathscr G=\bigcup_{s\in S} s\times X_{s^*s}.
\end{equation}
and define the groupoid structure: 
$(s, ty)\cdot(t,y)=(st ,y)$, 
and $(s,x)^{-1}=(s^*, s x)$. One can check that these operations are well-defined and $\mathcal G$ is a groupoid.  However this groupoid is too large for us and we need to invoke the equivalence of germs. 

The \emph{groupoid of germs} $\mG =S\ltimes X$ is defined (with an abuse of notation) as $\mathscr G$ modulo the equivalence relation $(s,x)\sim (t,y)$ if $x=y$ and there exists an idempotent $e$ such that  $x\in X_e$ and $se=re$. We denote the equivalence class of $(s,x)$ by $[s,x]$ and call it the germ of $s$ at $x$. It is a routine exercise to show that $\mG$ with $[s,ty][t,y]=[st,y]$ and $[s,x]\inv = [s^*,sx]$ is in fact a groupoid. 
Note that if $S$ is a group, then there are no identifications and we retrieve the transformation groupoid of Example~\ref{exphyhyh2}. 

 When $X$ is a Hausdorff topological space, one can show that 
 $[s\times U] := \{[s,x]\mid x\in U\}$, where $U\subseteq X_{s^*s}$ is open, 
 is a basis for a topology on $\mG$.   With this topology, $\mG$ is \'etale and
  $[s\times U]$ is an open bisection.  If $X$ has a basis of compact open sets, then $\mG$ is Hausdorff ample.    Further by \cite[Proposition~6.2]{exel20081} if $S$ is a semilattice and $X_e$ are clopen for $e\in E(S)$, then $\mG$ is Hausdorff.

\end{example}

\begin{example}[\bf{Underlying groupoid of an inverse semigroup}]\label{underhnyfuhr5}
Let $S$ be an inverse semigroup. The maps 
\begin{align*}
\dom: S &\longrightarrow E(S) &  \ran: S&\longrightarrow E(S)\\
s &\longmapsto s^*s  & s &\longmapsto ss^*
\end{align*}
considered as the source and range maps make $S$ into a groupoid with the product of the semigroup as the composition of the groupoid. The unit space is $E(S)$. Note that if $S$ is graded inverse semigroup, so is the underlying groupoid of $S$ and the strongly graded property passes from one structure to another. 
\end{example}


\subsection{Inverse semigroup of bisections of a groupoid} \label{hausfgrtgeud7}

Given an ample Hausdorff groupoid,  the inverse semigroup made up of all the compact open bisections plays an important role. In fact, the Steinberg algebra associated to a Hausdorff ample groupoid is the inverse semigroup ring of compact open bisections modulo their unions (see Definition~\ref{defsteinberg}). In the following,  we describe this inverse semigroup for a graded topological groupoid. Both the grading and the topology can be stripped away. 

Let $\mG$ be a $\Gamma$-graded Hausdorff ample groupoid.  Set 
\begin{equation}\label{jhfgyt7595}
\mG^{h}=\{U\; |\; U \text{~is a graded compact open bisection of~} \mG\}.
\end{equation} 
Then $\mG^{h}$ is an inverse semigroup under the multiplication $U.V=UV$ and inner inverse $U^*=U^{-1}$ as in (\ref{pofgtryhf}) and (\ref{pofgtryhf2})  (see \cite[Proposition 2.2.4]{paterson}).  Furthermore, the map $c:\mG^{h} \backslash \emptyset \rightarrow \Gamma, U\mapsto \gamma$, if $U\subseteq \mG_\gamma$, makes $\mG^{h}$ a graded inverse semigroup with $\mG^h_{\gamma}=c^{-1}(\gamma)$, $\gamma \in \Gamma$, as the graded components. 
 Observe that in  the inverse semigroup $\mG^{h}$, $B\leq C$ if and only if $B\sub C$ for $B, C\in \mG^{h}$. If from the outset we consider $\mG$ as a trivially graded groupoid, then we have an inverse semigroup consisting of all compact open bisections. In this case we denote the inverse semigroup by $\mG^a$. There are other notations for this semigroup in literature, such as $\mathcal S(\mG)$ in \cite{exel20081} or $\mG^{co}$ in \cite{paterson}. 

There is a natural action of inverse semigroup $\mG^h$ on the $\mGo$. In fact the groupoid of germs (as in Example~\ref{exphyhyh7}) of this action is $\mG$ itself and this allows us to relate the partial crossed product of construction to the concept of Steinberg algebras (see Theorem~\ref{thmpsisr}). We describe next this action. In fact in what follows we will construct a homomorphism of semigroups $\pi: \mG^h \rightarrow \mathcal I(\mGo)$. 

For each $B\in \mG^{h}$, $BB^{-1}$ and $B^{-1}B$ are compact open subsets of $\mGo$. Define 
\begin{align}
\pi_B: B^{-1}B &\longrightarrow B B^{-1}\label{fhfhgjg8585}\\
u &\longmapsto r(Bu)\notag
\end{align}
Since $B$ is a bisection, $Bu$ consists of only one element of $\mG$ and thus the map $\pi_B$ is well defined. 
 Observe that $\pi_B$ is a bijection with inverse $\pi_{B^{-1}}$. We claim that $\pi_B$ is a homeomorphism for each $B\in \mG^{h}$. Take any open subset $O\sub U_B$. Observe that $\pi_B^{-1}(O)=d(r^{-1}(O)\cap B)$ is an open subset of $U_{B^{-1}}$. Thus $\pi_B$ is continuous. Similarly, $\pi_B^{-1}$ is continuous. One can check that for compact open bisections $B$ and $C$, $\pi_B \pi _C =\pi _{BC}$, and thus $\pi: \mG^h \rightarrow \mathcal I(\mGo)$ is a homomorphism of inverse semigroups. If the grade group $\G$ is considered to be trivial, then we have a homomorphism $\pi: \mG^a \rightarrow \mathcal I(\mGo)$. This homomorphism is injective if, in some sense, there isn't too much \emph{isotropy} which we show in Lemma~\ref{campfresh} after introducing some more terminologies. 
 
 We say a topological groupoid $\mG$ is \emph{effective} if the interior of the isotropy bundle is just the unit space, that is
 \[\operatorname{Iso}(\mG)^{\circ} = \mGo.\]  
 Thus in an effective ample groupoid, if we have a compact open bisection $B$ such that every element $\gamma \in B$ has the property $s(\gamma)=r(\gamma)$, then $B \subseteq \mGo$. 
 
 We say a subset $U$ of the
unit space $\mG^{(0)}$ of $\mG$ is \emph{invariant} if $d(\g)\in U$ implies $r(\g)\in U$;
equivalently,
\[
    r(d^{-1}(U))=U=d(r^{-1}(U)).
\]

 For an invariant $U \subseteq \mG^{(0)}$,  we write $\mG_{U} := d^{-1}(U)$ which
coincides with the restriction \[\mG|_{U}=\{x\in\mG \mid  d(x)\in U,  r(x)\in U\}.\] Notice
that  $\mG_{U}$ is a groupoid with unit space $U$.

 We say $\mG$ is  \emph{strongly effective} if for every nonempty closed invariant subset $D$ of $\mG^{(0)}$, the groupoid $\mG_{D}$ is effective. These assumptions play important roles when classifying ideals of  Steinberg algebras (see \S\ref{dhfdujjdefjde3}). 
 
\begin{lemma}\label{campfresh}
Let $\mG$ be an ample groupoid. Then the morphism $\pi : \mG^h \rightarrow \mathcal I(\mGo)$ is an injective if and only if $\mG$ is effective.
\end{lemma}

For more equivalences of effective groupoids, see \cite[Lemma 3.1]{bcfs}. 

\subsection{Graph groupoids}\label{gfgfgfg1}
Our next goal is to describe groupoids associated to directed graphs. There is a general construction of a groupoid from a topological space $X$ and a local homeomorphism $\sigma:X \rightarrow X$, called a \emph{Deaconu-Renault groupoid} (see~\cite{re2000}). The graph groupoids are a special case. We briefly recall this general construction. 

Let $\sigma:X \rightarrow X$ be a local homeomorphism. Consider 
\begin{equation}\label{reangroup}
\mG(X,\sigma)=\{(x,m-n,y) \mid m,n \in \mathbb N, \sigma^m(x)=\sigma^n(y) \},
\end{equation}
with the groupoid structure inherited from the transitive groupoid $X\times \mathbb Z \times X$.  Note that $\mG(X,\sigma)$ is not transitive in general. 

When $X$ is a Hausdorff space, sets of the form 
\[Z(U,m,n,V)=\{(x,m-n,y)\mid (x,y) \in U\times V, \sigma^m(x)=\sigma^n(y) \},\] where $U$ and $V$ are open subsets of $X$ are a basis for a topology on $\mG(X, \sigma)$ making it a Hausdorff \'etale groupoid.  When $X$ also has a basis of compact open sets, the groupoid is Hausdorff ample. 

 To any graph $E$ one can associate a groupoid $\mG_E$, called the \emph{boundary path groupoid}, which we will just call the  \emph{graph groupoid} of $E$. This is the groupoid that relates the Steinberg algebras to the subject of Leavitt path algebras, as it's foundation is to relate graph $C^*$-algebras and groupoid $C^*$-algebras. To be precise, one can show there is a $\mathbb Z$-graded $*$-isomorphism $A_R(\mG_E) \cong L_R(E)$ (see Example~\ref{lpaexam12}). 

Let $E=(E^{0}, E^{1}, r, s)$ be a directed graph (see Example~\ref{ingfgfgfhr1}). We denote by $E^{\infty}$ the set of
infinite paths in $E$ and by $E^{*}$ the set of finite paths in $E$. Set
\[
X := E^{\infty}\cup  \{\mu\in E^{*}  \mid   r(\mu) \text{ is not a regular vertex}\}.
\]
Let
\[
\mG_{E} := \big \{(\a x,|\a|-|\b|, \b x) \mid   \a, \b\in E^{*}, x\in X, r(\a)=r(\b)=s(x) \big \}.
\]
We view each $(x, k, y) \in \mG_{E}$ as a morphism with range $x$ and source $y$. The
formulas $(x,k,y)(y,l,z)= (x,k + l,z)$ and $(x,k,y)^{-1}= (y,-k,x)$ define composition
and inverse maps on $\mG_{E}$ making it a groupoid with 
\[\mG_{E}^{(0)}=\{(x, 0, x) \mid x\in X\},\] which we will identify with the set $X$.

Next, we describe a topology on $\mG_{E}$ which is  ample and Hausdorff. For $\mu\in E^{*}$ define
\[
Z(\mu)= \{\mu x \mid x \in X, r(\mu)=s(x)\}\subseteq X.
\]
For $\mu\in E^{*}$ and a finite $F\subseteq s^{-1}(r(\mu))$, define
\[
Z(\mu\setminus F) = Z(\mu) \setminus \bigcup_{\a\in F} Z(\mu \a).
\]
The sets $Z(\mu\setminus F)$ constitute a basis of compact open sets for a locally
compact Hausdorff topology on $X=\mG_{E}^{(0)}$ (see \cite[Theorem 2.1]{we}).

For $\mu,\nu\in E^{*}$ with $r(\mu)=r(\nu)$, and for a finite $F\subseteq E^{*}$ such
that $r(\mu)=s(\a)$ for $\a\in F$, we define
\[
Z(\mu, \nu)=\{(\mu x, |\mu|-|\nu|, \nu x) \mid  x\in X, r(\mu)=s(x)\},
\]
and then
\[
Z((\mu, \nu)\setminus F) = Z(\mu, \nu) \setminus \bigcup_{\a\in F}Z(\mu\a, \nu\a).
\]
The sets $Z((\mu, \nu)\setminus F)$ constitute a basis of compact open bisections for a
topology under which $\mG_{E}$ is a Hausdorff ample groupoid.

In the case of the graph groupoid $\mG_E$, the topological assumptions on $\mG_E$ can be described in terms of the geometry of the graph $E$. We collect them here.

\begin{thm}\label{gfhtyfy6749}
Let $E$ be a directed graph and $\mG_E$ the graph groupoid associated to $E$. We have the followings: 
\begin{enumerate}[\upshape(1)]

\item The unit space $\mGo_E$  is finite if and only if $E$ is a no exit finite graph.  \cite{steinbergchain}

\medskip

\item The unit space $\mGo_E$  is compact if and only if $E$ has finite number of vertices.

\medskip 

\item The unit space $\mGo_E$ is topologically transitive if and only if $E$ is downward directed. \cite{steinbergprime}

\medskip

\item  The unit space $\mGo_E$ is effective if and only if $E$ satisfies condition (L).  \cite{steinbergprime}
\end{enumerate}
\end{thm}

The following table summarises the properties of the graph $E$ and the corresponding properties of the graph groupoid $\mG_E$.  

\begin{center}
\begin{tabular}{|l|l|}
\hline
{\bf Graph $E$ Property \qquad \qquad} &  {\bf  Groupoid $\mG_E$ property \qquad \qquad}\\
\hline
no cycles &  principle (istropy trivial)\\
\hline
condition (L) & effective\\
\hline
condition (K) & strongly effective \\
\hline
cofinal  &  minimal \\
\hline
$E^0$ finite  & $\mG^{(0)}$ compact\\
\hline
$E$ finite and no cycles & $\mG_E$ discrete\\
\hline
\end{tabular}
\end{center}

\section{Steinberg algebras}\label{stein4}

\subsection{Steinberg algebras}

Steinberg algebras (for Hausdorff ample groupoids) are universal algebras that can be defined in terms of inverse semigroup algebras.  We present the details below and then provide a concrete realisation as a convolution algebra consisting of certain continuous functions with compact support. In the last section we describe the algebra when the groupoid is not Hausdorff.

Recall that if $R$ is a commutative ring with unit, then the \emph{semigroup algebra} $RS$ of an inverse semigroup $S$ is defined as the $R$-algebra with basis $S$ and multiplication extending that of $S$ via the distributive law. If $S$ is an inverse semigroup with zero element $z$, then the \emph{contracted} semigroup algebra is $R_0S = RS/Rz$.
For a $\Gamma$-graded groupoid $\mG$, recall the graded inverse semigroup $\mG^h$ from \S\ref{hausfgrtgeud7}. 

\begin{deff} \label{defsteinberg} 
Let $\mG$ be a $\Gamma$-graded Hausdorff ample groupoid  with the inverse semigroup $\mG^h$. Given a commutative ring $R$ with identity, the \emph{Steinberg $R$-algebra} associated to
$\mG$, and denoted $A_R(\mG)$, is the contracted semigroup algebra $R_0\mG^h$, modulo the ideal generated by $B + D - B \cup D$, where $B,D,B\cup D \in \mG^h_\g$, $\g \in \Gamma$ and $B\cap D = \emptyset$.  
\end{deff}
So the Steinberg algebra $A_R(\mG)$, is the algebra generated by the set $\{t_B \mid B \in \mG^h \}$ with coefficients in $R$, subject to the relations
\begin{enumerate}
\item[(R1)]\label{it1:unst} $t_{\emptyset}=0$;
\smallskip
\item[(R2)]\label{it2:unst} $t_{B}t_{D}=t_{BD}$, for all $B, D\in \mG^h$; and
\smallskip
\item[(R3)]\label{it3:unst} $t_{B}+t_{D}=t_{B\cup D}$, whenever $B$ and $D$ are disjoint elements of
    $\mG^h_{\gamma}$, $\gamma \in \Gamma$, such that $B\cup D$ is a bisection.
\end{enumerate}
Thus, the Steinberg algebra is universal with respect to the relations (R1), (R2) and (R3) in that if $A$ is any algebra containing  $\{T_B:B \in \mG^h\}$ satisfying  (R1), (R2) and (R3), then there is a homomorphism from $A_R(\mG)$ to $A$ that sends $t_B$ to $T_B$. The uniqueness theorems would tell us when this natural homomorphism is injective~(\S\ref{pofhjfyt84}).

\begin{example}[\bf{Classical groupoid algebras}]
If $\mG$ is a groupoid and $A$ is a ring then $A$ is said to be a $\mG$-graded ring if $A=\bigoplus_{\gamma \in \mG} A_{\gamma}$, where $A_{\gamma}$ is an additive subgroup of $A$ and $A_\beta A_\gamma \subseteq A_{\beta \gamma}$, if $(\beta, \gamma) \in \mG^{(2)}$ and $A_\beta A_\gamma=0$, otherwise. A prototype example of $\mG$-graded rings are classical groupoid algebras which we describe next. Let $R$ be a ring. Let $R\mG$ be a free left $R$-module with basis $\mG$, {\it i.e.,} 
$R\mG=\bigoplus_{\gamma \in \mG} R \gamma$, where $R \gamma=R$.  We define a multiplication as follows: 
\[\sum_{\sigma \in \mG} r_\sigma \sigma \, . \,  \sum_{\tau \in \mG} s_\tau  \tau=\sum_{\sigma, \tau \in \mG}   r_\sigma s_\tau \sigma \tau, \]
when $(\sigma, \tau) \in \mG^{(2)}$, and $0$ otherwise. This makes $R\mG$ an associative ring and setting $(R\mG)_\gamma=R\gamma$, clearly gives this ring a $\mG$-graded structure. 

It is not difficult to see that if $\mGo$ is finite, then $R\mG$ is a direct summand of matrix rings over corresponding isotropy group rings as follows:  Let $O_1, \cdots, O_k$ be the orbits of $\mGo$. Note that for $x,y \in O_i$, there is an isomorphism between the isotropy groups $\mG_x^x \cong \mG_y^y$. Choosing $x_i \in \O_i$, $1\leq i \leq k$, we then have 
\begin{equation}\label{camaug12}
R\mG \cong \bigoplus_{i=1}^{k} \M_{n_i}(R G_i),
\end{equation}
where $G_i=\mG_{x_i}^{x_i}$ and $n_i=|O_i|$.

Now if $\mG$ has a discrete topology, one can easily establish that $A_R(\mG)\cong R\mG$. On the other hand, for the case of \'etale groupoid, finiteness of $\mG^{(0)}$ implies $\mG$ is discrete (Lemma~\ref{fghfhfhuyt}). Putting these together  we have the following.

\begin{prop}\label{gfhgdhr22} \cite[Proposition~3.1]{steinbergchain}
Let $\mG$ be an ample groupoid with $\mG^{(0)}$ finite. Let $O_1, \cdots, O_k$ be the orbits of $\mGo$ and let $G_i$ be isotropy group of $O_i$ and $n_i=|O_i|$.  Then 
\begin{equation*}
A_R(\mG) \cong \bigoplus_{i=1}^{k} \M_{n_i}(R G_i),
\end{equation*}

\end{prop}

Consider the set $I=\{1,\dots, n\}$, $n \in \mathbb N$ and $G=\{e\}$ a trivial group. Then the transitive groupoid $I\times I$ (see Example~\ref{exphyhyh}) with discrete topology is ample. Proposition~\ref{gfhgdhr22} now immediately gives
\[A_R(I\times I )\cong \M_n(R).\]

\end{example}

\begin{example}[{\bf Leavitt path algebras}]\label{lpaexam12}

Let $E$ be a graph. The Leavitt path algebra associated to the graph $E$ was introduced as a purely  algebraic version of graph $C^*$-algebras. We refer the reader to the book \cite{AAS} for a general introduction to the theory and \cite{rigby} for an excellent survey on the connection of these algebras with Steinberg algebras. We briefly give an account of how to model Leavitt path algebras as Steinberg algebras. 

For a graph $E$, let $\mG_E$ be the associated graph groupoid (see \S\ref{gfgfgfg1}).  By \cite[Example~3.2]{cs} 
the map
\begin{align}
\label{assteinberg}
\pi_{E} : L_{R}(E) & \longrightarrow A_{R}(\mG_{E}),\\
\mu\nu^{*}-\sum_{\a\in F}\mu\a\a^{*}\nu^{*} &\longrightarrow 1_{Z((\mu,\nu)\setminus F)} \notag
\end{align}
extends to a $\mathbb Z$-graded algebra isomorphism.  Observe that the
isomorphism of algebras in \eqref{assteinberg} satisfies
\begin{equation}
\label{valuation}
\pi_{E}(v)=1_{Z(v)}, \quad\pi_{E}(e)=1_{Z(e, r(e))}, \quad\pi_{E}(e^{*})=1_{Z(r(e), e)},
\end{equation}
for each $v\in E^{0}$ and $e\in E^{1}$.

If $w : E^1 \to \G$ is a function, we extend $w$ to $E^{*}$ by
defining $w(v) = 0$ for $v \in E^0$, and $w(\a_{1}\cdots \a_{n}) = w(\a_{1})\cdots
w(\a_{n})$. Thus $L_R(E)$ is a $\Gamma$-graded ring. On the other hand, defining $\widetilde{w}:\mG_{E}\xra \G$ by 
\begin{equation}\label{poiuytre}
    \widetilde{w}(\a x, |\a|-|\b|, \b x) = w(\a) w(\b)^{-1}, 
\end{equation}
gives a cocycle (\cite[Lemma 2.3]{kp}) and thus $A_R(\mG)$ is a $\Gamma$-graded ring as well. A quick inspection of isomorphism (\ref{assteinberg}) shows that $\pi_E$ respects the $\Gamma$-grading.

\end{example}

\subsection{Convolution algebra of continuous functions from $\mG$ to $R$}  \label{hf7hgj55}
In this subsection we give an alternative definition for Steinberg algebras. Let $A$ be the algebra $C(\mG)$ of continuous functions from $\mG$ to $R$ (where $R$ is viewed as a discrete space) which is locally constant and has compact support.  Here addition and scalar multiplication are defined pointwise and multiplication is given by \emph{convolution} where
\[f*g(\gamma) = \sum_{\alpha\beta=\gamma} f(\alpha)g(\beta).\]
Since $\mG$ is ample and Hausdorff, for each $B \in \mG^h$, $B$ is clopen so  the characteristic function $\chi_B$ (where $\chi_B(\gamma) = 1$ for $\gamma \in B$ and 0 otherwise) is continuous.  For $B,D \in \mG$, one can check that $\chi_B*\chi_D = \chi_{BD}$ and the set of characteristic functions
\[\{\chi_B : B \in \mG\} \subseteq C(\mG)\]
satisfies the relations (R1), (R2) and (R3).  So the universal property gives
us a homomorphism from $A_R(\mG)$ to $C(\mG)$ that takes $t_B$ to $\chi_B$ for $B \in \mG^h$. The range of this homomorphism is the subalgebra \[\operatorname{span} \{\chi_B : B \in \mG^h\} = \{f \in C(\mG) \mid f \text{ is locally constant and has compact support}\} .\]   Further this homomorphism is injective by \cite[Theorem~6.3]{st}  and hence an isomorphism.
Thus one can write 
\begin{align*} 
A_R(\mG )& = \text {span } \{1_B \mid  B \text{ is a homogeneous compact open bisection} \},\\
& = \Big \{\sum_{B\in F} r_B 1_B \mid  F: \text{ mutually disjoint finite collection of homogeneous compact open bisections} \Big \}, \\
& \text{ addition and scalar multiplication of functions are pointwise},\\
& \text{ multiplications on the generators are }  1_B 1_D = 1_{BD}.
\end{align*}

\subsection{Centre of a Steinberg algebra}

The centre of Steinberg algebras were determined in \cite{st} and it has a very pleasant description. We describe it here. 

We say $f \in A_R(\mG)$  is a \emph{class function} if $f$ satisfies the following conditions:
\begin{enumerate}
\item if $f(x) \not =0$ then $s(x)=r(x)$;

\medskip 

\item  if $s(x) = r(x) = s(z)$ then  $f(zxz^{-1}) = f(x)$.
\end{enumerate}

Here we consider $f\in C(\mG) $ as a continuous function on $\mG$ and $x \in \mG$. 

\begin{prop}\cite[Proposition 4.13]{st}
The centre of $A_R(\mG)$ is the set of class functions.
\end{prop}

Note that if $f$ is a class function, then $\supp(f)\subseteq \Iso(\mG)^{\circ}$. Thus if  $\mG$  is effective, then the centre is contained in the diagonal subalgebra $A_R(\mG^{(0)})$. The diagonal preserving isomorphisms play an important role in realising groupoids from the algebra isomorphisms (see Theorem \ref{jdyr6bdfhe6e}). 

\subsection{Uniqueness theorems}\label{pofhjfyt84}

A \emph{uniqueness theorem} gives criteria under which a homomorphism from the Steinberg algebra to another $R$-algebra is injective.  Uniqueness theorems are useful when studying other concrete realisations of Steinberg algebras.  The most general uniqueness theorem is the following which is \cite[Theorem~3.1]{cep}:

\begin{thm}
\label{unthm}
Let $\mG$ be a second-countable, ample, Hausdorff groupoid and let $R$ be a unital commutative ring. Suppose that $A$ is an $R$-algebra and that $\pi: A_R(\mG) \to A$ is a ring homomorphism. Then $\pi$ is injective if and only if $\pi$ is injective on $A_R(\operatorname{Iso}(\mG)^o)$, the subalgebra generated by elements of $\mG^h$ that are also contained in the interior of the isotropy bundle.  
\end{thm}

Theorem~\ref{unthm} has the assumption of second-countability because the proof requires the unit space to be a `Baire space'.  The graded uniqueness theorem, below (which is \cite[Theorem~3.4]{clarkmichell} ) does not have this assumption.  Instead it requires a particular graded structure.  

\begin{thm}
\label{grunthm}
Let $\mG$ be a Hausdorff, ample groupoid, $R$ a commutative ring with identity, $\Gamma$ a discrete group, and 
$c : \mG \to \Gamma$ a continuous functor such that $\mG_e$ is effective. 
Suppose  $\pi : A_R(\mG) \to A$ is a graded ring homomorphism. 
Then $\pi$ is injective if and only if $\pi(rt_K) \neq 0$ for every nonzero $r \in R$ and compact open $K \subseteq \mG^{(0)}$.
\end{thm}

\subsection{Ideal structures of Steinberg algebras} \label{dhfdujjdefjde3}

There is a satisfactory description of ideals of a Steinberg algebra $A_k(\mG)$ based on the geometry of the groupoid $\mG$, where the algebra is over a field $k$ (so the ideals of the coefficient ring does not interfere) and the groupoid is Hausdorff ample and (strongly) effective.

The first result is the simplicity of these algebras. 

\begin{thm}\cite{bcfs}\label{jhsd76hbdfh}
Let $\mG$ be an Hausdorff, ample groupoid, and $k$ a field. Then $A_k(\mG)$ is simple if and only if $\mG$ is effective and 
$\mG^{(0)}$ has no open invariant subsets.
\end{thm}

A glance at the table of properties of the graph versus the graph groupoids (\S\ref{gfgfgfg1}) shows that Theorem~\ref{jhsd76hbdfh} is parallel to the first theorem proved in the theory of Leavitt path algebras, namely,  
for an arbitrary graph $E$, the Leavitt path algebra $L_k(E)$ is simple if and only if $E$ satisfies condition (L) and $E^0$ has no nontrivial saturated hereditary subsets~(\cite{abrams2005}, \cite[\S2.9]{AAS}). 

For an invariant $U\subseteq \mG^{(0)}$, one can easily see that the set 
 \[I(U) := \text{span} \{1_B \mid  s(B) \subseteq  U\},\]
is an ideal of $A_k(\mG)$. In fact, if the groupoid is strongly effective, this is the only way one can construct ideals in these algebras. 

\begin{thm} \cite{chuef} 
Suppose $\mG$ is a strongly effective ample groupoid. Then the correspondence
\[U \longmapsto I(U),\]
is a lattice isomorphism from the lattice of open invariant subsets of $\mG^{(0)}$ onto the lattice of ideals of $A_k(\mG)$.
\end{thm}

\begin{thm}\cite{cep}
Let $\mG$ be a $\Gamma$-graded ample groupoid such that $\mG_\varepsilon$  is strongly effective. Then the correspondence
\[U \longmapsto I(U),\]
is an isomorphism from the lattice of open invariant subsets of $\mG^{(0)}$ onto to the lattice of graded ideals in $A_k(G)$.
\end{thm}

We refer the reader to \cite{chuef,cep} for further results on the ideal theory of Steinberg algebras. 

\section{Combinatorial and dynamical invariants of \'etale groupoids}

There are several ``combinatorial'' invariants one can associate to a groupoid such as the full group and homology groups. For certain groupoids these combinatorial invariants  are related to  very interesting Higman-Thompson groups or $K$-groups as we describe next in this section. 

\subsection{Full groups}
For an \'etale groupoid $\mG$ with the compact unit space $\mGo$, the full group $\FG$ was defined by Matui~\cite{matui2012}. The full group 
$[[\mathbb Z \times X]]$ of the transformation groupoid of the action of $\mathbb Z$ on a Cantor set $X$ via a minimal homeomorphism  (see Example~\ref{exphyhyh}) coincides with the full group defined and studied in \cite{giordano}. We define the full group of a groupoid here and collect results related to this group. 

Recall that $\mG^a$ is the inverse semigroup of compact open bisections and $\pi: \mG^a \rightarrow \mathcal I(\mGo)$ the action of $\mG^a$ on $\mGo$ (\S\ref{hausfgrtgeud7}). Let $G:=\{U \in \mG^a \mid \dom(U)=\ran(U)=\mGo \}$ be a subgroup of $\mG^a$ (here $\mGo$ considered to be compact). Then the \emph{full group of $\mG$}, denoted by $\FG$ is $\pi(G)$. In fact for a noncompact $\mGo$, one can give a generalised version of this notion and define the \emph{full inverse semigroup} by 
$\FG=\pi(\mG^a)$. 

If the Hausdorff ample groupoids $\mG$ and $\mH$ are isomorphism, then clearly $A_R(\mG)\cong  A_R(\mH)$. Further, since from the outset, $\mG \cong \mH$  induces $\mGo \cong \mHo$, we have the \emph{diagonal isomorphism} $A_R(\mGo)\cong  A_R(\mHo)$ as well. Renault in \cite{renaultcartan} established the converse of this statement for certain groupoid $C^*$-algebras. Several recent papers progressively improved this result in the algebraic setting (see \cite{aradia, cardia} and \cite{steinbergdia}). Combining with the full group invariant, we have the following theorem, relating groupoids, inverse semigroups and algebras.

\begin{thm}\label{jdyr6bdfhe6e}
Let $R$ be a unital commutative ring without nontrivial idempotents and let $\mG$ and $\mH$ be Hausdorff effective ample groupoids. Then the following are equivalent.
\begin{enumerate}[\upshape(1)] 
\item  $\mG$ and $\mH$ are isomorphic;
\medskip
\item  the inverse semigroups $\mG^a$ and $\mH^a$ are isomorphic;
\medskip

\item  the inverse semigroups $[[G]]$ and $[[H]]$ are isomorphic;

\medskip
\item  there exists a diagonal-preserving isomorphism between the Steinberg algebras $A_R(\mG)$ and $A_R (\mH)$. 

\end{enumerate}
\end{thm}

Consider a graph with one vertex and two loops \[E:\, \, \, \, \, \, \, \, \, \, \,
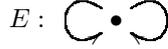
\xymatrix{\bullet  \ar@(lu,ld)  \ar@(ur ,dr) }\] 

\smallskip

\noindent  and its graph groupoid $\mG_E$ as described in \S\ref{gfgfgfg1}. The unit space $\mG_E^{(0)}$ is compact (Theorem~\ref{gfhtyfy6749}) and we can consider the full group $[[\mG_E]]$. On the other hand, let $L_k(E)$ be the Leavitt path algebra associated to $E$ over a field $k$. An element $a\in L_k(E)$ is called unitary if $a a^*=a^* a =1$. Consider the set 
\[P_{2,1}=\Big \{ a \in L_k(E) \text{ is unitary} \mid  a=\sum_{i=1}^l \alpha_i \beta_i ^* \Big \}, \]
where $\alpha_i,\beta_i$ are distinct paths in $E$ and the coefficients in the sum are all $1$. One can prove that 
\[P_{2,1} \cong [[\mG_E]],\] and they are isomorphic to the Thompson group $T_{2,1}$, which was constructed in 1965 and was the first infinite finitely presented simple group.  In fact, considering a graph with one vertex and $n$ loops, we retrieve Higman-Thompson groups $G_{n,1}$ constructed by Higman in 1974. We refer the reader to \cite{pardo,soren} for this line of research.

\subsection{Homology and K-theory}

The homology theory for \'etale groupoids was introduced by Crainic and Moerdijk~\cite{crainicmoerdijk} 
who showed these groups are invariant under Morita equivalences of \'etale groupoids and established some spectral sequences which used for the computation of these homologies. Matui \cite{matui2012,matui2015,matui2016} considered this homology theory in relation with the dynamical properties of groupoids and their full groups.   In~\cite{matui2012} Matui proved, using Lindon-Hochschild-Serre spectral sequence established by Crainic and Moerdijk that for an \'etale groupoid $\mG$ arising from subshifts of finite type, the homology groups $H_0(\mG)$ and $H_1(\mG)$ coincide with $K$-groups $K_0(C^*(\mG))$ and $K_1(C^*(\mG))$, respectively.  Here $C^*(\mG)$ is the groupoid $C^*$-algebra associated to $\mG$ which were first systematically studied by Renault in his seminal work~\cite{renault}. In the language of graphs, Matui proved that for a finite graph $E$ with no sinks
\begin{equation}\label{gftrgftrg}
K_0(C^*(E)) \cong H_0 (\mG_E) \, \, \text{   and,    }  K_1(C^*(E)) \cong H_1 (\mG_E). 
\end{equation}

In this section we recall the construction of the homology of an ample groupoids and recount Matui's conjecture relating $K$-groups of groupoid $C^*$-algebras to the homology of its groupoid (Conjecture~\ref{popokjhgfd}). It would be very desirable to establish a relation between the homology groups and $K$-groups of Steinberg algebras, as here we have higher $K$-theories available and strong $K$-theory machinery which works on them~\cite{quillen}.  

Let $X$ be a locally compact Hausdorff space and $R$ a topological abelian group.  Denote by $C_c(X,R)$  the set of $R$-valued continuous functions with compact support. With point-wise addition, $C_c(X,R)$ is an abelian group. 
Let $\pi:X\to Y$ be a local homeomorphism 
between locally compact Hausdorff spaces.  
For $f\in C_c(X,R)$, define the map $\pi_*(f):Y\to R$ by 
\[
\pi_*(f)(y)=\sum_{\pi(x)=y}f(x). 
\]
 Thus $\pi_*$ is a homomorphism from $C_c(X,R)$ to $C_c(Y,R)$ which makes $C_c(-,R)$  a functor from the category of locally compact Hausdorff spaces with local homeomorphisms to the category of abelian groups.

Let $\mG$ be an \'etale groupoid. For $n\in\N$, we write $\mG^{(n)}$ 
for the space of composable strings of $n$ elements in $\mG$, that is, 
\[
\mG^{(n)}=\{(g_1,g_2,\dots,g_n)\in\mG^n\mid
s(g_i)=r(g_{i+1})\text{ for all }i=1,2,\dots,n{-}1\}. 
\]
For $i=0,1,\dots,n$, with $n\geq 2$
we let $d_i:\mG^{(n)}\to \mG^{(n-1)}$ be a map defined by 
\[
d_i(g_1,g_2,\dots,g_n)=\begin{cases}
(g_2,g_3,\dots,g_n) & i=0 \\
(g_1,\dots,g_ig_{i+1},\dots,g_n) & 1\leq i\leq n{-}1 \\
(g_1,g_2,\dots,g_{n-1}) & i=n. 
\end{cases}
\]
When $n=1$, we let $d_0,d_1:\mG^{(1)}\to\mG^{(0)}$ be 
the source map and the range map, respectively. 
Clearly the maps $d_i$ are local homeomorphisms.

Define the homomorphisms $\partial_n:C_c(\mG^{(n)},R)\to C_c(\mG^{(n-1)},R)$ 
by 
\begin{equation} \label{hdtfjdjjd}
\partial_1=s_*-r_* \text{~~~~and~~~~}
\partial_n=\sum_{i=0}^n(-1)^id_{i*}. 
\end{equation}
One can check that the sequence 
\begin{equation}\label{moorcomlex}
0\stackrel{\partial_0}{\longleftarrow}
C_c(\mG^{(0)},R)\stackrel{\partial_1}{\longleftarrow}
C_c(\mG^{(1)},R)\stackrel{\partial_2}{\longleftarrow}
C_c(\mG^{(2)},R)\stackrel{\partial_3}{\longleftarrow}\cdots
\end{equation}
is a chain complex of abelian groups.

The following definition comes from \cite{crainicmoerdijk,matui2012}.

\begin{deff}\label{homology}({\it Homology groups of a groupoid $\mG$})
Let $\mG$ be an \'etale groupoid. Define the homology groups of $\mG$ with coefficients $R$, $H_n(\mG,R)$, $n\geq 0$, to be the homology groups of the Moore complex~(\ref{moorcomlex}), 
i.e., $H_n(\mG,R)=\ker\partial_n/\Ima\partial_{n+1}$. 
When $R=\Z$, we simply write $H_n(\mG)=H_n(\mG,\Z)$. 
In addition, we define 
\[
H_0(\mG)^+=\{[f]\in H_0(\mG)\mid f(x)\geq0\text{ for all }x\in\mG^{(0)}\}, 
\]
where $[f]$ denotes the equivalence class of $f\in C_c(\mG^{(0)},\Z)$. 
\end{deff}

Extending Matui's result (\ref{gftrgftrg}), in \cite{hazli} using the description of monoid of Leavitt path algebras, it could be proved that for any graph (with sinks, source and infinite emitters), we have 
\begin{equation*}
H_0(\mG_E) \cong K_0 (A(\mG_E)) \cong K_0 (L(E)) \cong K_0 (C^*(E))\cong K_0 (C^*(\mG_E)).
\end{equation*}

Before stating Matui's conjecture (Conjecture~\ref{popokjhgfd}) we also state a class of groupoids that the zeroth homology $H_0$ coincides with Grothendieck group $K_0$ and their algebras fall into Elliott's class of algebras that can be classified by $K_0$-groups. 

Let $\mG$ be a second countable \'etale groupoid whose unit space is compact and totally disconnected. Then the subgroupoid $\mH \subseteq  \mG$ is an 
\emph{elementary subgroupoid} if $\mH$ is a compact open principal subgroupoid of $\mG$ such that $\mHo = \mGo$ The groupoid $\mG$ is called an \emph{AF groupoid} if it can be written as an increasing union of elementary subgroupoids.

If  $\mG$ is an AF groupoid, then Steinberg algebra $A_R(\mG)$, for a field $R$, is an ultramatricial algebra (or the reduced groupoid $C^*$-algebra $C^*(\mG)$ is an AF algebra and thus the terminology). For such groupoids, there is an order-preserving isomorphism 
\begin{align*}
\pi: H_0(\mG) &\longrightarrow K_0(A_R(\mG)),\\ 
[1_{\mGo}] &\longmapsto [1_{A_R(\mGo}].
\end{align*}

We have then the following theorem. 
\begin{thm}
Let $R$ be a field and $\mG$ and $\mH$ are AF groupoids. Then the following are equivalent.
\begin{enumerate}[\upshape(1)] 
\item  $\mG$ and $\mH$ are isomorphic;
\medskip
\item There is an ordered preserving isomorphism $H_0(\mG) \cong H_0(\mH)$ which sends $[1_{\mGo}]$ to $[1_{\mHo}]$;

\medskip
\item  there exists a $R$-algebra isomorphism between the Steinberg algebras $A_R(\mG)$ and $A_R (\mH)$. 

\end{enumerate}
\end{thm}

This theorem points to a direction which is gaining ever more importance of finding a class of \'etale groupoids that a variant of $K$-theory and homology theory can be a complete invariants. 

Recall that an \'etale groupoid $\mG$ is said to be effective if the interior of its isotropy coincides with its unit space $\mGo$ and minimal if every orbit is dense.

The following conjecture of Matui (\cite[Conjecture 2.6]{matui2016}) expresses the $K$-theory of a groupoid $C^*$-algebra as a direct sum of homology groups of the associated groupoid. For one thing, this indicates that the homology groups provide much finer invariants than the $K$-groups.

\begin{conjecture}[Matui's HK Conjecture]\label{popokjhgfd}
 Let $\mG$ be a locally compact Hausdorff \'etale groupoid such that $\mGo$ is a Cantor set. Suppose that $\mG$ is both effective and minimal. Then 

\begin{equation}\label{matuiconj0}
K_0(C^*_r(\mG)) \cong \bigoplus_{i=0}^{\infty} H_{2i}(\mG)
\end{equation}
\begin{equation}
K_1(C^*_r(\mG)) \cong \bigoplus_{i=0}^{\infty} H_{2i+1}(\mG) 
\end{equation}

\end{conjecture}

Apart from (arbitrary graphs) Matui proved this conjecture for AF groupoids with compact unit space and in \cite{matui2016} for all finite Cartesian products of groupoids associated to shifts of finite type. Ortega also showed in \cite{ortega} that the conjecture is valid for the Katsura-Exel-Pardo groupoid 
$\mG_{A,B}$ associated to square integer matrices with $A  \geq 0$.

\section{Partial crossed product rings}\label{cross4}

In this section we consider the ``partially'' group ring like algebras and relate them to Steinberg algebras. This is also another demonstration how the inverse semigroups and algebras arising from them are related to groupoids and algebras coming from them. 

Let $\pi=(\pi_s, A_s, A)_{s\in\SS}$ be a partial action of the inverse semigroup $\SS$ on an algebra $A$. Here $A_s \subseteq A$, $s\in \SS$ is an ideal of $A$ and $\pi_s: A_{s^*}\xra A_s$ an isomorphism such that for all $s, t \in\SS$ 

\begin{itemize}
\item[(i)] $\pi^{-1}_s=\pi_{s^*}$;

\item[(ii)] $\pi_s(A_{s^*}\cap A_t)\sub A_{st}$;

\item[(iii)] if $s\leq t$, then $A_s\sub A_t$;

\item[(iv)] For every $x\in A_{t^*}\cap A_{t^*s^*}$, $\pi_s(\pi_t(x))=\pi_{st}(x)$.
\end{itemize}

This is a generalisation of the concept of partial group actions (see~\S\ref{exphyhyh5}). Define $\LL$ as the set of all formal forms $\sum_{s\in \SS}a_s\d_s$ (with finitely many $a_s$ nonzero), where $a_s\in A_s$ and $\d_s$ are symbols, with addition defined in the obvious way and multiplication being the linear extension of $$(a_s\d_s) (a_t\d_t)=\pi_s\big(\pi_{s^{-1}}(a_s)a_t\big)\d_{st}.$$ Then $\LL$ is an algebra which is possibly not associative. Exel and Vieira proved under which condition $\LL$ is associative (see \cite[Theorem 3.4]{exelvieira}). In particular, if each ideal $A_s$ is idempotent or non-degenerate, then $\LL$ is associative (see \cite[Theorem 3.4] {exelvieira} and \cite[Proposition 2.5]{de}). 
This algebra is too large for us and we need to consider this ring modulo idempotents, as follows. 
Consider $\mathcal{N}=\langle a\d_s-a\d_t: a\in A_s, s\leq t\rangle$, which is the ideal generated by $a\d_s-a\d_t$. The \emph{partial skew inverse semigroup ring} $A\rtimes_{\pi}\SS$ is defined as the quotient ring $\LL/ \mathcal{N}$.

Next we equip these algebras with a graded structure. Suppose $\SS$ is a $\Gamma$-graded inverse semigroup (see \S\ref{unfairjfgng}). 
Observe that the algebra $\LL$ is a $G$-graded algebra with elements $a_s\d_s\in\LL$ with $a_s\in A_s$ are homogeneous elements of degree $w(s)$. Furthermore, if $s\leq t$, then $s=ts^*s$. It follows that $w(s)=w(t)w(s^*)w(s)=w(t)$. Hence $a\d_s-a\d_t$ with $s\leq t$ and $a\in A_s$ is a homogeneous element in $\LL$. Thus the ideal $\mathcal N$ generated by homogeneous elements is a graded ideal and therefore the quotient algebra $A\rtimes_{\pi}\SS= \LL/\mathcal N$ is $\Gamma$-graded.

Let $X$ be a Hausdorff topological space and $R$ a unital commutative ring with a discrete topology. Let $C_R(X)$ be the set of $R$-valued continuous function (i.e., locally constant) with compact support (see also \S\ref{hf7hgj55}).  If $D$ is a compact open subset of $X$, the characteristic function of $D$, denoted by $1_D$, is clearly an element of $C_R(X)$. In fact, every $f$ in $C_R(X)$ may be written as 
 \begin{equation}\label{rep}
 f=\sum_{i=1}^nr_i1_{D_i}, 
 \end{equation} where $r_i\in R$ and the $D_i$ are compact open, pairwise disjoint subsets of $X$. $C_R(X)$ is a commutative $R$-algebra with pointwise multiplication. The support of $f$, defined
by $\supp(f)=\{x\in X\;|\; f(x)\neq 0\}$, is clearly a compact open subset.

We observe that $C_R(X)$ is an idempotent ring. We have \begin{equation}\label{equationforlocal}
 \sum_{i=1}^n1_{D_i} \cdot f=f \cdot \sum_{i=1}^n1_{D_i}=\sum_{i=1}^nr_i1_{D_i}=f
 \end{equation} for any $f\in C_R(X)$ which is written as \eqref{rep}. So $C_R(X)$ is a ring with local units and thus an idempotent ring.

For $\Gamma$-graded Hausdorff ample  groupoid $\mG$, recall the inverse semigroup $\mG^{h}$ from (\ref{jhfgyt7595}) and the 
action of $\mG^h$ on $\mGo$ from (\ref{fhfhgjg8585}). 
There is an induced action $(\pi_B, C_R(BB^{-1}), C_R(\mGo))_{B\in \mG^{h}}$ of $\mG^{h}$ on an algebra $C_R(\mGo)$. Here the map $\pi_B: C_R(B^{-1}B)\xra C_R(BB^{-1})$ is given by $\pi_B(f)=f\circ \pi_B^{-1}$. We still denote the induced action by $\pi$. In this case, $\LL=\big \{\sum_{B\in \mG^{h}}a_B\d_B\;|\; a_B\in C_R(BB^{-1})\big \}$ is associative, since each ideal $C_R(BB^{-1})$ is idempotent. Since $\mG^{h}$ is $\Gamma$-graded, $C_R(\mGo)\rtimes_{\pi} \mG^{h}$ is a $\Gamma$-graded algebra. 

We are in a position to relate partial skew inverse semigroup rings to Steinberg algebras.

\begin{thm} \label{thmpsisr} Let $\mG$ be a $\G$-graded Hausdorff ample groupoid and $\pi=(\pi_B, C_R(BB^{-1}), C_R(\mGo))_{B\in\mG^{h}}$ the induced action of $\mG^{h}$ on $C_R(\mGo)$. Then there is a $\Gamma$-graded isomorphism of $R$-algebras
\begin{equation}
A_R(\mG)\cong_{\gr}C_R(\mG^{(0)})\rtimes_{\pi}\mG^{h}.
\end{equation} 
\end{thm}
\begin{proof}
For each $D\in \mG^{(h)}$, define 
\[t_D=1_{r(D)}\delta_D  \in C_R(\mG^{(0)})\rtimes_{\pi}\mG^{h}.\]

One can check that the set $\{t_D\; | \; D \in\mG^{(h)}\}$ satisfies (R1), (R2) and (R3) relations in the Definition \ref{defsteinberg} of Steinberg algebras. Thus we obtain a homomorphism $f: A_R(\mG)\rightarrow C_R(\mG^{(0)})\rtimes_{\pi}\mG^{h}.$

 Next define a map $g: C_R(\mG^{(0)}) \rtimes_{\pi}\mG^{(h)}\xra A_R(\mG)$. For each 
 $B\in \mG^{(h)}$ and $a_B\d_B\in \mathcal{L}$, we define 
 \begin{equation*}
g(a_B\d_B)=
\begin{cases}
a_B(r(x)), & \text{~if~} x\in B,\\
0, & \text{otherwise}.
\end{cases}
\end{equation*}
One can check that the map $g$ is well-defined and $gf=\id_{A_R(\mG)}$ and $fg=1_{C_R(\mG^{(0)})\rtimes_{\pi}\mG^{h}}$. 
\end{proof}

We refer the reader to \cite{beutergoncalves,beuter,hazli} for more results relating the partial inverse semigroup algebras to Steinberg algebras and \cite{exel20081} for the $C^*$-versions of these results. 

\section{Non-Hausdorff ample groupoids}
\label{sec:nonH}

We finish this paper with a brief discussion about non-Hausdorff groupoids. 
When relaxing the Hausdorff assumption on $\mG$, we still insist that the unit space $\mGo$ be Hausdorff so that, in the setting of \'etale and ample groupoids, they are locally Hausdorff.  With this weakened hypothesis the universally defined Steinberg algebra of Definition~\ref{defsteinberg} no longer works.  However, we can still study the Steinberg algebra of such a groupoid using one of the other characterisations, such as 
\[A_R(\mG)= \operatorname{span} \{\chi_B : B \in \mG^h\}.  \]
Since $\mG$ is not Hausdorff,  compact open bisections are no longer closed and hence characteristic functions might not be continuous.  

Other fundamental results in the theory of Steinberg algebras fail for non-Hausdorff groupoids, for example the Uniqueness theorems (see \cite[Example~3.5]{cep}).  
Still, progress is slowly being made to develop a theory.  
Recall that in an ample groupoid $\mGo$ is always open in $\mG$.  It turns out that $\mGo$ is closed in $\mG$ if and only if $\mG$ is Hausdorff.  So an important step for understanding non-Hausdorff groupoids is  to understand the closure of $\mGo$.

\section{Acknowledgments}
The authors would like to acknowledge Australian Research Council grant DP160101481 and  Marsden grant 15-UOO-071 from the Royal Society of New Zealand.

\end{document}